%% file: main.tex
\documentclass{article} 
\usepackage[utf8]{inputenc}
\usepackage[T1]{fontenc}

\usepackage{paper}

\usepackage[round]{natbib}
\makeatletter
\newcommand{\TODO}[1]{\textcolor{red}{[TODO\@ifnotempty{#1}{: #1}]}}
\makeatother

\title{A Fast Spectral Algorithm for Mean Estimation with Sub-Gaussian Rates}
\author{Zhixian Lei\thanks{Harvard John A. Paulson School of Engineering and Applied Sciences
(SEAS), Harvard University. Email: \texttt{\{zhixianlei,pvenkat\}@g.harvard.edu}.} \and Kyle
Luh\thanks{Center of Mathematical Sciences and Applications, Harvard University. Email:
\texttt{kyleluh@gmail.com}.} \and Prayaag Venkat\samethanks[1] \and Fred Zhang\thanks{Department
of Electrical Engineering and Computer Sciences, UC Berkeley. Email: \texttt{z0@berkeley.edu}. Part of the work was done at Harvard University.}}

\date{}
\usepackage[top=1in,bottom=1in,left=1in,right=1in,includefoot,
  footskip=30pt]{geometry}
\allowdisplaybreaks
\begin{document}

\maketitle

\begin{abstract}%
   We study the algorithmic problem of estimating the mean of a heavy-tailed random vector in
    $\Real^d$, given $n$ i.i.d.\ samples. The goal is to design an efficient estimator that
    attains the optimal sub-gaussian error bound, only assuming that the random vector has
    bounded mean and covariance. Polynomial-time solutions to this
    problem are known but have high runtime due to their use of semi-definite programming (SDP). Moreover, conceptually, it   
    remains open whether convex relaxation is truly necessary for this
    problem.
    
    In this work, we show that it is possible to go beyond SDP and achieve better
    computational efficiency. In particular, we provide a \textit{spectral}  algorithm that achieves the optimal statistical performance and runs in time
    $\widetilde O\left(n^2 d \right)$, improving upon the previous fastest runtime $\widetilde O
    \left(n^{3.5}+ n^2d\right)$ by Cherapanamjeri \etal\  (COLT '19). Our algorithm is spectral  in that it only
    requires (approximate) eigenvector computations, which can be implemented very efficiently
    by, for example, power iteration or the Lanczos method. 

    At the core of our algorithm is a novel connection between the furthest hyperplane problem
    introduced by  Karnin \etal\ (COLT '12) and a structural lemma on heavy-tailed distributions
    by Lugosi and Mendelson (Ann.\ Stat.\ '19). This allows us to iteratively reduce the
    estimation error at a geometric rate using only the information
    derived from the
    top singular vector of the data matrix, leading to a significantly faster running time.\end{abstract}

\newpage
\input{intro.tex}
\input{prelim.tex}
\input{gd.tex}
\input{inner_max.tex}

\input{main_algo.tex}
 \input{discuss.tex}

\section*{Acknowledgements}
{The authors would like to thank Boaz Barak and Jelani Nelson for helpful conversations. In particular, we would like to thank Boaz for directing us to the paper~\cite{barak2009uniform}. 

Zhixian Lei and Prayaag Venkat have been supported by NSF awards CCF 1565264 and CNS 1618026, and the Simons Foundation. Kyle Luh has been partially supported by NSF postdoctoral fellowship DMS-1702533.  Prayaag Venkat has also been supported by an NSF Graduate Fellowship under grant DGE1745303.}

\bibliography{bib}

\appendix
\input{code.tex}

\input{apx_fact.tex}

\input{omit.tex}
\input{karnin.tex}
\end{document}

%% file: intro.tex
\section{Introduction}\label{sec:intro}
Estimating the mean of a multivariate distribution from samples is among the most fundamental statistical problems. Surprisingly, it was only recently that a line of works in the statistics literature culminated in an estimator achieving the optimal statistical error under minimal assumptions~(\cite{lugosi2019mean}).
However, from an algorithmic point of view, computation of this estimator appears to be intractable. 
On the other hand, fast estimators, such as the empirical average, tend to achieve sub-optimal statistical performance. 
The following question  remains open:
\begin{center}
    \textit {Can we provide  simple, fast algorithm that computes a statistically optimal mean estimator in high dimensions, under minimal assumptions?}
\end{center} 
In this paper, we make progress towards this goal, under the classic setting where only finite mean and covariance are assumed. Formally, our problem is defined as follows.  Given $n$ i.i.d.\ copies $\bX_1, \ldots, \bX_n$ of a random vector $\bX \in \R^d$ with bounded mean $\bmu = \E \bX$ and covariance $\bSig = \E (\bX - \bm{\mu})(\bX - \bm{\mu})^T$, compute an estimate $\widehat{\bm{\mu}} = \widehat{\bm{\mu}}(\bX_1, \ldots, \bX_n)$ of  the mean $\bm{\mu}$. Our goal is to show that for any failure probability $\delta \in (0,1]$,
\[
\Pr \left( \|\widehat{\bm{\mu}} - \bm{\mu} \| > r_{\delta} \right) \leq \delta,
\]
for as small a  radius $r_{\delta}$ as possible. Moreover, we would like to compute $\widehat{\bmu}$ efficiently. The \naive\ estimator is simply the empirical mean
$
\overline{\bmu} = \frac{1}{n} \sum_{i=1}^n \bX_i$.
It is well known that among all estimators, the empirical mean minimizes mean squared error. However, if we instead use the size of the deviations to quantify the quality of the estimator, the empirical mean is only optimal for sub-gaussian random variables (\cite{catoni2012empirical}).  When $\bX \sim \mathcal{N}(\bmu, \bSig)$ we have with probability at least $1- \delta$,
\begin{equation}
\|\overline{\bm{\mu}} - \bm{\mu} \| \leq \sqrt{\frac{\Tr (\bSig)}{n}} + \sqrt{\frac{2 \| \bSig \| \log (1/ \delta)}{n}}
\end{equation}
An estimator that achieves  above is said to have \emph{sub-gaussian performance} or  \emph{sub-gaussian rate}. 

In practical settings, assuming that the samples obey a Gaussian distribution may be unrealistic. In an effort to design robust estimators, it is natural to study the mean estimation problem under very weak assumptions on the data.
A recent line   of works~(\cite{catoni2012empirical,minsker2015geometric,devroye2016sub,joly2017estimation,lugosi2018near,lugosi2019mean}) study the mean estimation problem when the samples obey a heavy-tailed distribution. 

For heavy-tailed distributions the performance of the empirical mean is abysmal. If we only assume that $\bm X$ has finite mean $\bmu$ and covariance $\bSig$, then by Chebyshev's inequality, the empirical mean only achieves error  of order $ \sqrt{ {\Tr (\bSig)}/{\delta n}}$, which is worse than the sub-gaussian rate in two ways. First, its dependence on $\frac{1}{\delta}$ is exponentially worse. Second, the $\Tr (\bSig)$ term, which may grow with the dimension $d$, is multiplied the dimension-independent term $\sqrt{{1}/{\delta n}}$, whereas in the Gaussian case, the two are separate.

\paragraph{Median-of-means paradigm}
Surprisingly, recent work has shown that it is possible to improve on the performance of the empirical mean using the \emph{median-of-means} approach. For $d = 1$, the following construction, originally due to~\cite{nemirovsky1983median, jerrum1986median, alon1999median}, achieves sub-gaussian performance:
\begin{enumerate}[(i)]
    \item First, bucket the data into $k = \lceil 10\log (1/\delta)\rceil$ disjoint groups and compute their means $Z_1,Z_2,\cdots, Z_k$.
    \item Then, output the median $\widehat{\mu}$ of $ Z_1, Z_2,\cdots, Z_k$.
\end{enumerate}

A long line of work has followed this paradigm and generalized it to higher dimensions~(\cite{catoni2012empirical,devroye2016sub,joly2017estimation, lugosi2018near,lugosi2019mean}). The key challenge is to correctly define a notion of median for a collection of points in $\R^d$. \cite{minsker2015geometric} considered $\widehat{\bm{\mu}}_{GM}$ defined to be the \emph{geometric median} of the bucket means $\bZ_1, \ldots, \bZ_k$. For some constant $c_{GM}$, with probability at least $1- \delta$, it satisfies
\begin{equation}\label{eq:gmintro}
\|\widehat{\bm{\mu}}_{GM} - \bm{\mu} \| \leq c_{GM} \sqrt{\frac{\Tr \bSig \cdot \log  (1/\delta)}{n}}.
\end{equation}
This achieves the correct dependence on $\delta$, but the dimension dependent and independent terms are still not separated. Following this work, \cite{lugosi2018near} described another estimator $\widehat{\bm{\mu}}_{LM}$ which finally achieved the optimal sub-gaussian radius. The idea behind their construction is to consider \textit{all $1$-dimensional projections} of the bucket means and try to find an estimate that is close to the median of the means of all projections. Formally, the estimator is given by
\begin{align}\label{eqn:est}
    \widehat{\bm \mu}_{LM}= \argmin_{\bm\x \in \R^{d}} \max_{\bm{u}\in \mathbb S^{d-1}}
    \left\lvert\text{median}\left\{\left\langle{\bm{Z}_i},
    \bm{u}\right\rangle\right\}_{i=1}^k -  \langle \bbx, \bu\rangle\right\rvert.
\end{align}
Clearly, searching over all directions in $\mathbb S^{d-1}$ requires exponential time. The key question, therefore, is whether one can achieve both computational and statistical efficiency simutaneously. 
\paragraph{Computational considerations}
A priori, it is  unclear that the Lugosi-Mendelson estimator can be computed in polynomial time as a direct approach involves solving an intractable optimization problem. Moreover, the Lugosi-Mendelson analysis seems to suggest that estimation in the heavy-tailed model is conceptually harder than under (adversarial) corruptions. In the latter, each sample can be classified as either an inlier or an outlier. In the heavy-tailed setting,  Lugosi-Mendelson shows that there is a majority of the bucket means that cluster around the true mean along any projection. 
However, a given sample may be an inlier by being close to the mean when projected onto one direction, but an outlier when projected onto another. In other words, the \emph{set} of inliers may change from one direction to another.

Surprisingly, a recent line of works have established the polynomial-time computability of Lugosi-Mendelson estimator.
\cite{hopkins2018sub} formulates $\widehat{\bmu}_{LM}$ as the solution of a low-degree polynomial optimization problem and showed that using the Sum-of-Squares SDP hierarchy to relax this problem yields a sub-gaussian estimator. While the  run-time of this algorithm is polynomial, it involves solving a large SDP. 
Soon after, \cite{cherapanamjeri2019fast} provided an iterative method  in which each iteration involves solving a smaller, explicit SDP, leading to a  run-time of $\widetilde O\left(n^{3.5} + n^2d\right)$\footnote{Throughout we use $\widetilde{O}(\cdot)$ to hide polylogarithmic factors (in $n$, $d$ and $\log (1/\delta)$).}.
Even more recently, a concurrent and independent work by \cite{lecue2019robust} gave an estimator with sub-gaussian performance that can be computed in time $\widetilde{O}(n^2 d)$. 
The construction is inspired by a near-linear time algorithm for robust mean estimation under
adversarial corruptions due to \cite{ge2018learning}. The algorithm requires solving (covering) SDPs. 

We note, however, that a common technique in these algorithms is SDP,
which tends to be impractical for large sample sizes and in high dimensions. In contrast, our algorithm only requires approximate eigenvector computations. For a problem as fundamental as mean estimation, it is desirable to obtain simple and ideally
practical solutions. A key conceptual message of our work is that SDP is indeed unnecessary and
can be replaced by simple spectral techniques.

\paragraph{Our result}
In this work,  we demonstrate for the first time that mean estimation with sub-gaussian rates can be achieved  efficiently \textit{without}  SDP.  The runtime of the algorithm matches the independent work of \cite{lecue2019robust}. 
In addition, our algorithm enjoys robustness against (additive)  corruptions, where the number of adversarial points is a small fraction of $k$.

It is known that there exists an information-theoretic requirement for achieving such rates---that is,  $\delta
\geq 2^{-O(n)}$~(\cite{devroye2016sub}).
Under this assumption, we give an efficient spectral algorithm.
\begin{theorem}\label{thm:most}
Let $\delta \geq A e^{-n}$ for a  constant $A$ and $k= \lceil 3600\log(1/\delta)\rceil$.
 Given $n$ points $\mathcal G\cup \mathcal B$, where $\mathcal G$ are i.i.d.\ samples from a distribution over $\mathbb R^d$ with mean $\bmu$ and covariance $\bSig$ and $\mathcal B$ a set of arbitrary points with $|\mathcal B| \leq k/200$, there is an efficient algorithm  that 
outputs an estimate $\widehat{\bmu} \in \R^d$ such that with probability at least $1 - \delta$,
\[
\| \bmu - \widehat{\bmu} \| \leq C \left( \sqrt{\frac{\Tr (\bSig)}{n}} + \sqrt{\frac{ \| \bSig \|
\log (1/ \delta)}{n}}\right),
\]
for a constant $C$. Furthermore, the algorithm runs in time $O\left(nd+ k^2 d \,\polylog (k,d)\right)$.
\end{theorem}
The algorithm is iterative.  Each iteration only requires an (approximate) eigenvector
computation, which can be implemented in nearly linear time by power iteration or the Lanczos
algorithm. We believe that our algorithm can be fairly practical.

\paragraph{Other related works} 
Recently,~\cite{prasad2019unified} established a formal connection between the Huber contamination model and the heavy-tailed model we study in this paper. They leverage this connection to use an existing $\widetilde{O}(n d^2)$-time mean estimation algorithm of \cite{diakonikolas2016robust}  to design estimators for the heavy-tailed model. Under moment assumptions, their estimator achieves performance better than   geometric median~\eqref{eq:gmintro}, yet worse than sub-gaussian. 

In addition, algorithmic robust statistics has gained much attention in the
theoretical computer science community in recent years. 
A large body of works have studied  the mean estimation problem with \emph{adversarially} corrupted samples, with the focus on providing efficient algorithms (\cite{diakonikolas2016robust,lai2016agnostic,ge2018learning,dong2019quantum}). For a more complete survey, see~\cite{robustsurvey}

Going beyond mean estimation, there has been a recent spate of works on other statistical
 problems under heavy-tailed distributions. We refer the
readers to~\cite{lugosi2019survey} for a survey. 

\paragraph{Technical overview}
Our main algorithm builds upon the iterative approach of
\cite{cherapanamjeri2019fast}. 
For simplicity, assume there is no adversarial point.
At a high level, for each iteration $t$, the
algorithm will maintain a current guess $\bbx_t$ of the true mean. 
To update, Cherapanamjeri~\etal~study  the inner maximization of $\widehat{\bmu}_{LM}$~\eqref{eqn:est} with $\bbx = \bbx_t$. They showed that under Lugosi-Mendelson structral condition, the problem is essentially equivalent of following program, which we call $\mathcal{M}(\bbx_t,\bZ)$:
\begin{align*}
\text{max} \quad & \theta \\
    \text{subject to} \quad 
    & b_i \ip{\bZ_i - \bbx_t}{\bu} \geq b_i \theta \text{ for } i = 1,\ldots, k\\
   & \sum_{i=1}^k b_i  \geq 0.95k\\
   & \bb \in \{0,1\}^k , \bu \in \mathbb S^{d-1}.
\end{align*}
It can be shown that an optimal solution $\bu \in \mathbb{S}^{d-1}$ will align with the unit vector in the direction of $\bmu - \bbx_t$, and $\theta$ approximates $\|\bmu - \bbx_t\|$. Hence, one can perform the update $\bbx_{t+1}  \leftarrow \bbx_t + \gamma \theta \bu$, for some appropriate constant $\gamma$, to geometrically decrease the distance of $\bbx_{t}$ to $\bmu$. 


In this work, we  start by drawing a  connection between the above program and the furthest hyperplane problem (FHP) of \cite{karnin2012unsupervised}.  This allows us to avoid the SDP approach in \cite{cherapanamjeri2019fast}. The problem can be formulated as the following:
\begin{align} \label{FHP} \tag*{(FHP)}
\text{max} \quad & \theta \\
    \text{subject to} \quad 
    &  |\ip{\bZ_i - \bbx_t}{\bu}| \geq  \theta \text{ for } i = 1,\ldots, k\label{eq:twoside}\\
   & \bu \in \mathbb S^{d-1}\nonumber.
\end{align}
In the original formulation due to Karnin \etal, the goal is to find a maximum margin linear
classifier for a collection of points, where the margin is \emph{two-sided}.  Notice that any
feasible solution to $\mathcal{M}(\bbx_t,\bZ)$ satisfies at least $0,95k$ constraints of FHP as well. For an arbitrary dataset, the two-sided margin requirement indeed provides a relaxation. One technical observation of this work is that it is \textit{not} a significant one, for the random data we care about---if a major fraction of the constraint~\eqref{eq:twoside} are satisfied, then most constraints of $\mathcal{M}(\x_t,\bZ)$ are  satisfied as well. 
 
 Unfortunately, the algorithm of Karnin \etal~cannot directly apply, as it only works under a strong promise that there exists a feasible solution that satisfies \emph{all} of the  constraints~\eqref{eq:twoside}. 
In our setting, there may not be such a feasible solution; we can only guarantee that there exists a unit vector (namely, the one in the direction of $\bmu - \bbx_t$) that satisfies most of constraints with large margin.

Our main contribution is to provide an algorithm that works even under this weak promise. We now briefly review the algorithm of Karnin \etal, show why it fails for our purpose, and explain how we address the issues that arise. Suppose that there exists a unit vector $\bu^*$ and $\theta^*$ which are feasible for  the FHP problem. Then, averaging the constraints tells us that 
\[
\frac{1}{k} \sum_{i=1}^k \ip{\bZ_i}{\bu^*}^2 \geq \theta^{*2}.
\]
Hence, if we define $\bu$ to be the top right singular vector of the matrix $\bA$ whose rows are $\bZ_i$, then
\[
\| \bA \bu \|_2^2 =  \sum_{i=1}^k \ip{\bZ_i}{\bu}^2 \geq  \sum_{i=1}^k \ip{\bZ_i}{\bu^*}^2 \geq k \theta^{*2},
\]
so  $\bu$ satisfies the constraints in~\ref{FHP} \emph{on average}. However, the distribution of the quantities $\ip{\bZ_i}{\bu}^2$ may be extremely skewed, so that $\bu$ only satisfies a few of the constraints with large margin. If this happens, however, we can downweight those constraints which are satisfied by $\bu$ with large slack to encourage it to satisfy more constraints. This reweighting procedure is repeated several times, and at the end we  use a simple rounding scheme to yield a single output vector with the desired properties from all the repetitions. In particular, this weighting scheme is essentially the same as the classic \emph{multiplicative weights update} (MWU) method~(\cite{arora2012multiplicative}) for regret minimization, as we show in Appendix~\ref{sec:fhp-mwu}.

If we are only guaranteed that $\bu^*$ satisfies most, but not all, of the constraints, then the inequality $\sum_{i=1}^k \ip{\bZ_i}{\bu^*}^2 \geq k \theta^{*2}$ may no longer hold when the points $\bZ_i$ get re-weighted and the algorithm of Karnin \etal~cannot be guaranteed to converge. To illustrate this point, consider the following extreme case. Suppose that after the first iteration, the algorithm finds the vector $\bu^*$ as the top right singular vector of $\bA$. In the re-weighting procedure, the constraints $i$ for which $\ip{\bZ_i}{\bu^*}^2 \geq \theta^{*2}$ may be down-weighted significantly, whereas the remaining constraints may be unaffected. This may result in most of the weight being concentrated on the constraints $i$ where $\ip{\bZ_i}{\bu^*}^2 \ll \theta^{*2}$. In the second iteration, we have no guarantee of the behavior of the top singular vector of the re-weighted matrix because all the weight is concentrated on a small set consisting of these ``bad'' constraints.

To address this scenario, our key technical idea is to project the weights onto the set of \textit{smooth distributions} after each update. Informally, the notion of smooth distribution enforces that no point can take too much probability mass---say, more than ${4}/{k}$. This prevents the weights from ever being concentrated on too small a subset and allows us to guarantee that $\sum_{i=1}^k \ip{\bZ_i}{\bu^*}^2 \geq k \theta^{*2}$ still holds approximately.
Moreover, the appropriate notion of projection here is that of a {Bregman projection}.
Leveraging our earlier MWU interpretation of the algorithm
(\autoref{sec:fhp-mwu}), we apply a classic regret bound for MWU under Bregman projection (\cite{arora2012multiplicative}), and this yields the same guarantee of the original algorithm.
Finally, we remark that the projection can be computed quickly. Combining all these ideas together, we  manage to bypass the barrier of having bad points, under the much weaker assumption on $\bu^*$.  

\paragraph{Organization} The remainder of this article is organized as follows. In~\autoref{sec:prelim}, we set up the notations and specify  assumptions on the data. In~\autoref{sec:gd}, we explain the high level approach based on an iterative descent procedure from~\cite{cherapanamjeri2019fast}. The procedure requires us to approximately maximize a (non-convex) objective, and we discuss its properties in~\autoref{sec:inner_max}. \autoref{sec:main-algo} contains the main technical innovations of this work, where we design and analyze a faster algorithm for the aforementioned optimization problem.

%% file: prelim.tex
\section{Preliminaries and Assumptions} \label{sec:prelim}
In the following, we use $r_\delta = \sqrt{{\Tr (\bSig)}/{n}} + \sqrt{{ \| \bSig \| \log (1/
\delta)}/{n}}$ to denote the optimal, sub-gaussian error rate and $k =\lceil 3200\log (8/\delta)\rceil$. 
The input data $\{\bX_i\}_{i=1}^n$  consist of $\mathcal G$, a set of i.i.d.\ points, and $\mathcal B$, a set of adversarial points, with $|B| \leq k/200$.
Our algorithm  preprocesses the data $\bX_i$   into the bucket
means $\bm{Z}_1,\bZ_2,\cdots, \bZ_{2k} \in \Real^d$.%
\footnote{We assume
$\delta$ is such that $k\leq n/2$; as we mentioned in the introduction, this is information-theoretically necessary, up to a constant~(\cite{devroye2016sub}).}
Let $\mathcal B_j$ be the set of $\bX_i$ in bucket $j$.
We say that a bucket mean $\bZ_j$ is \textit{contaminated} if $B_j$ contains an adversarial $X_i \in B$ and \textit{uncontaminated} otherwise. 
Note that the number of contaminated bucket means is at most $k/200$.

Our argument is built on the Lugosi-Mendelson condition. It states that under any one-dimensional projection, most of the (uncontaminated) bucket means are close to the true mean, by an additive factor of $O(r_\delta)$. 
Throughout, we pessimistically assume all contaminated bucket means do not satisfy this property (under any projection) and condition on the following event.
\begin{assumption}[Lugosi-Mendelson condition]\label{asp:key}
Under the setting above, for all unit $\bm{v}$, we have
\begin{equation*}
   \left | \left\{ i :    \ip{\bv}{\bZ_i} - \ip{\bv}{\bmu}    \geq 600 r_\delta\right\}  \right| \leq 0.05k.
\end{equation*}
\end{assumption}

\begin{lemma}[\cite{lugosi2019mean}]\label{lem:lmm}
\autoref{asp:key} holds with probability at least $1-\delta  /8$.
\end{lemma}

%% file: gd.tex
\section{Descent Procedure} \label{sec:gd}
At a high level, our algorithm builds upon the  iterative descent paradigm of \cite{cherapanamjeri2019fast}. It maintains a sequence of estimates and updates via distance and gradient estimate.
\begin{definition}[distance estimate]\label{def:distest}
    We say that $d_t$ is a distance estimate (with respect to $\x_t$) if
    \begin{enumerate}[(i)]
        \item  when $\|\bmu - \x_t \|\leq 14000r_\delta$, we have  $d_t
\leq 28000r_\delta$; and
        \item when $\| \bm{\mu} - \x_t\| > 14000r_\delta$, we have
                \begin{equation} \label{cond:dist-est}
                 \frac{1}{21}\| \bm{\mu} - \x_t\|\leq d_t   \leq  2 \| \bm{\mu} - \x_t\|
                \end{equation}
    \end{enumerate}
\end{definition}
\begin{definition}[gradient estimate]\label{def:gradest}
    We say that $\bg_t$ is a gradient estimate (with respect to $\x_t$) if
        \begin{equation} \label{cond:grad-est}
            \ip{\bg_t}{\frac{\bm{\mu} - \x_t}{\| \bm{\mu} - \x_t\|}} \geq \frac{1}{200}  
        \end{equation}
        whenever $\| \bm{\mu} - \x_t\| > 14000r_\delta$.
   \end{definition}
\begin{figure}[htbp]
    \centering\begin{mdframed}[style=algo] 
\begin{enumerate}
    \item \textbf{Input:} Buckets means $\bZ_1,\ldots,\bZ_k \in \R^{d}$, initial estimate
        $\x_0$, iteration count $T_{\textsf{des}}$, and step size $\eta$.
    
    \item For $t=1,\ldots,T_{\textsf{des}}$:
        \begin{enumerate}
            \item Compute $d_t = \textsc{DistEst}(\bZ', \x_t)$.
            \item Compute $\bg_t = \textsc{GradEst}(\bZ', \x_t)$.
            \item Update $\x_{t+1} = \x_t + \eta d_t \bg_t$.
        \end{enumerate}
        
    \item \textbf{Output:} $\x_{t^*}$, where $t^* = \argmin_{t} d_t$.
\end{enumerate}
\end{mdframed}
\caption{Main algorithm---\textsc{Descent}}
\label{alg:descent}
\end{figure}

Suppose we initialize the estimate with coordinate-wise median-of-means which achieves an error rate $   \sqrt{\| \bSig\| kd/n}$ (\cref{lem:gm}).  The following lemma states that if \textsc{DistEst} and \textsc{GradEst} provide distance
and gradient estimate, then the algorithm \textsc{Descent} succeeds in logarithmic iterations. 
The lemma has essentially appeared in~\cite{cherapanamjeri2019fast}, albeit with a
general initialization and a 
different  set of constants. We give a proof in Appendix~\ref{sec:gdo} for completeness. 
\begin{lemma}[convergence rate; see \cite{cherapanamjeri2019fast}] \label{lem:descent-thm}
    Assume  that for all $t \leq T_{\textsf{des}}$, $d_t$ is a distance estimate and $\bg_t$ is
    a gradient estimate (with respect to $\x_t$).  Suppose $\| \bmu - \bbx_0 \| \leq
    O\left(\sqrt{\| \bSig\| kd/n}\right)$. Then the output of Algorithm~\ref{alg:descent}
    \textsc{Descent} instantiated with $T_{\textsf{des}}= \Theta
\left( \log d \right)$ and $\eta = {1}/{8000}$ satisfies
$    \| \x_{t^*} - \bm{\mu} \| \leq O\left(r_\delta\right)$.
\end{lemma} 

%% file: inner_max.tex
\section{Inner Maximization and its Two-Sided Relaxation}\label{sec:inner_max}

\cite{cherapanamjeri2019fast} obtains gradient and distance estimates by  solving  the \textit{inner maximization} problem of the Lugosi-Mendelson estimator, denoted by $\mathcal{M}(\bbx, \bZ)$:
\begin{align*}
\text{max} \quad & \theta \\
    \text{subject to} \quad 
    & b_i \ip{\bZ_i - \x}{\bw} \geq b_i \theta \text{ for } i = 1,\ldots, k\\
   & \sum_{i=1}^k b_i  \geq 0.95k\\
   & \bb \in \{0,1\}^k , \bw \in \mathbb S^{d-1}.
\end{align*}
We also denote  its feasibility  version for a fixed $\theta$ by $\mathcal{M}(\theta, \bbx,
\bZ)$. Note that the constraint of $\mathcal{M}( \bbx, \bZ)$  dictates that  $0.95$ fraction of the data must lie on \textit{one} side of the hyperplane $\bw$ with a margin $\theta$.
As discussed in the introduction, we relax it by allowing a two-sided margin:
$\mathcal{M}_2(\bbx, \bZ)$.
\begin{align*}
\text{max} \quad & \theta \\
    \text{subject to} \quad 
    & b_i |\ip{\bZ_i - \x}{\bw}| \geq b_i \theta \text{ for } i = 1,\ldots, k\\
   & \sum_{i=1}^k b_i  \geq 0.95k\\
   & \bb \in \{0,1\}^k , \bw \in \mathbb S^{D-1}.
\end{align*}
One technical observation here is that under the Lugosi-Mendelson condition, this relaxation is insignificant. Indeed,  approximately solving the problem suffices for gradient and distance estimates.  

\begin{lemma}\label{cor:estimate}
Let $\theta^*$ be the optimal value of $\mathcal{M}(\x, \bZ)$ and $\bw$ 
be a unit vector such that for at least $k/8$ of the $\bZ_i$, we have $|  \ip{\bw}{\bZ_i -\x}  | \geq \theta$,
where $\theta =  0.1\theta^*$.
We have that (i) $\theta$ is a distance estimate and
(ii) either $\bw$ or $-\bw$ is a gradient estimate.
\end{lemma}
We give a proof in Appendix~\ref{sec:oinner_max}. The intuition here is simple. If $\|\x-\bmu\|\ll r_\delta$, then the Lugosi-Mendelson condition ensures at most $0.05k$ points are far from $\x$ by $O(r_\delta)$ (under any projection), so $\theta =O(r_\delta)$. On the other hand, if $\|\x-\bmu\|\gg r_\delta$, along the gradient direction, a majority of data lie only on  \textit{one} side of the hyperplane, the side that contains the true mean, so the two-sided constraint does not make a difference.

%% file: main_algo.tex
\section{Approximating the Inner Maximization}\label{sec:main-algo}
We now give an algorithm that efficiently computes a approximate solution to the relaxation of the inner maximization.
This will provide  gradient and distance
estimates for each iteration of the main \textsc{Descent} algorithm (Algorithm~\ref{alg:descent}).

 The run-time of the algorithm is proportional $1/\theta^2$. For technical reasons, we need to ensure that $ \| \bZ_i -\x \| \leq 1$ for all $i$. However, \naive ly scaling all the data would decrease $\theta$, thereby blowing up the running time. Hence, as a preprocessing step, we prune out a small fraction of points $\bZ_i -\x$ with large norm before scaling. 

\subsection{Pruning and scaling}\label{sec:prune}
The  preprocessing step (Algorithm \ref{alg:prune}) will be
executed \textit{only once} in the algorithm.
After the pruning step and an appropriate scaling, we may assume the following structures on the data.
\begin{assumption}\label{ass:scale}
Given a current estimate $\x$, the pruned dataset $\bZ\in \R^{k'\times d}$ of size $k'$, let $\bZ'_i= \tfrac{1}{B}\left(\bZ_i - \x\right)$, where $B = \max_i \| \bZ_i' - \x \|$. We assume (i) $\|\bZ_i'\|\leq 1 $; 
    (ii) $k'\geq 0.9k$; and 
    (iii)  there exists  $\theta=\Omega(1/\sqrt{d})$ and a unit vector $\bw$
        such that for at least $0.8k$ points $|\ip{  \bZ_i' }{\bw}| \geq   \theta$.
\end{assumption}
We analyze the subroutine and prove the following lemma in Appendix~\ref{apx:prune}.
\begin{lemma}\label{cor:2ass}
With probability at least $1-\delta/8$, \autoref{ass:scale} holds for any  $\x$ such that $
\|\x -\bmu\|\leq  O \left(\sqrt{ {\|\bSig\| kd}/{n} }\right)$ and $\|\x -\bmu\|
\geq \Omega(r_\delta)$.
\end{lemma}

\begin{figure}[htbp]
    \centering\begin{mdframed}[style=algo] 
\begin{enumerate}
    \item \textbf{Input:}  Dataset $\bZ_1,\bZ_2,\cdots, \bZ_k \in \R^d$, initial estimate $\x_0$
         \item Compute the distances $d_i = \| \bZ_i - \x_0\|$.
         \item Sort the points  by   $d_i$ in decreasing order. 
         \item Remove the top $1/10$ fraction of them. Let $\bZ_1, \cdots, \bZ_{k'}$ be the remaining data.
    \item \textbf{Output:} $\bZ_1, \cdots, \bZ_{k'}$
\end{enumerate}
\end{mdframed}
\caption{\textsc{Prune}}
\label{alg:prune}
\end{figure}

In the remainder of the section, given a current estimate $\x$, we work with the pruned and scaled data, centered at $\x$, which we call $\bZ'\in \R^{k'\times d}$.

We will aim at proving the following lemma, under~\autoref{ass:scale}.  
\begin{lemma}[key lemma] \label{lem:main}  
    Assume~\autoref{ass:scale}. Let $\delta \in (0,1)$ and $T_{\textsf{des}}= \Theta(\log d)$. 
Suppose that there exists $\bw^* \in \mathbb S^{d-1}$  which satisfies  $|\ip{\bZ'_i
}{\bw^*}| \geq  \theta^*$ for $0.8k$ points in $\{\bZ_i'\}$.
Then there is an algorithm \textsc{ApproxBregman} which, with probability at least $1-
{\delta}/4T_{\textsf{des}}$, outputs $\bw \in \mathbb S^{d-1}$ such that for at least $0.45$
fraction of the points $\bZ'_i $, 
it holds that $|\ip{\bZ'_i  }{\bw}| \geq 0.1 \theta^*$. 

Further, \textsc{ApproxBregman} runs in time $\widetilde O\left(k^2 d\right)$.
\end{lemma}


\subsection{Approximation via Bregman Projection} \label{sec:breg}

In this section, we give the main algorithm for approximating $\mathcal{M}_2$. 
Suppose (by binary search) that we know the optimal margin $\theta$ in~\autoref{lem:main}. 
The goal is to find a unit vector $\bw$ such that a constant fraction of $\bZ'_i$ has margin $|\ip{\bZ'_i }{\bw}| \geq \theta$. 
The intuition is that we can start by computing the top singular vector of $\bZ'$.
Then the margin would be large on average: 
certain points may overly satisfy the margin demand while other may under-satisfy it.
Hence, we would downweight those data poitns that achieve large margin and compute the top singular vector of the weighted matrix again.

However,  it may stop making progress if it
puts too much weight on the  points that do not satisfy the margin bound. 
In this section, we show how to  prevent this scenario from occurring. 
The key idea is that at every iteration,  we ``smooth'' the weight vector $\tau_t$ so that we
can guarantee progress is being made.  We will formulate our
algorithm in the well-studied regret-minimization framework and appeal to existing
machinery~(\cite{arora2012multiplicative}) to derive the desired approximation
guarantees. 

First, we define what type of distribution we would like $\tau_t$ to
be. 

\begin{definition}[Smooth distributions]
The set of \emph{smooth distributions} on $[k']$ is defined to be
\[
\mathcal{K} = \left\{p \in \Delta_{k'} : p(i) \leq \frac{4}{k'} \text{ for every } i \in [k'] \right\},
\]
where $\Delta_{k'}$ is the set of probability distributions on $[k']$,
\[
\Delta_{k'} =\left\{p: [k'] \rightarrow [0,1] : \sum_{i \in [k']} p(i) = 1\right\}.
\]
\end{definition}

In the course of the algorithm, after updating $\tau_t$ as in the previous section, it may no longer be smooth. Hence, we will replace it by the closest smooth weight vector (under KL divergence). The following fact confirms that finding this closest smooth weight vector can be done quickly.

\begin{fact}[\cite{barak2009uniform}]
For any $p \in \Delta_k$ with support size at least ${k'}/{2}$, computing
\[
\Pi_{\mathcal{K}} (p) = \argmin_{q \in \mathcal{K}} KL(p||q)
\]
can be done in $\Tilde{O}(k')$ time, where $\text{KL} (\cdot || \cdot)$ denotes the Kullback-Leibler divergence.
\end{fact}
\begin{remark}
In our algorithm, we will only compute Bregman projections of distributions of
support size   at least ${k'}/{2}$. This is because neither our reweighting procedure nor the actual projection algorithm of~\cite{barak2009uniform} sets any coordinates to $0$ and the initial weight is uniform.
\end{remark}
\begin{figure}[ht]
    \centering\begin{mdframed}[style=algo] 
\begin{enumerate}
    \item \textbf{Input:} Buckets means $\bZ'  \in \R^{k' \times d}$,   margin $\theta$,
        iteration count $T \in \N$
    
    \item Initialize weights:  $\bm\tau_1 = \frac{1}{k'} (1, \ldots,1) \in \R^{k'}$.
    
    \item For $t=1,\ldots,T$, repeat: 
        \begin{enumerate}
            \item Let $\bA_t$ be the $k' \times d$ matrix whose $i$th row is $\sqrt{\bm\tau_t(i)} (\bZ'_i )$ and $\bw_t$  be its approximate top right singular vector .
            
            \item Set $\sigma_t(i) = |\ip{\bZ'_i  }{\bw_t}|$.
            
            \item Reweight: If $\|\bA_t \bw_t \|_2^2 \geq \frac{\theta^2}{10}$, then
                $\bm\tau_{t+1}(i) =\bm \tau_t(i) \left(1 -\bm\sigma_t(i)^2/2\right)$ for $i \in [k']$. Otherwise, do not change the weights.
            
            \item Normalize: Let $Z = \sum_{i \in [k']} \bm\tau_{t+1}(i)$ and redefine $\bm\tau_{t+1} \leftarrow \frac{1}{Z} \bm\tau_{t+1}$.
            
            \item Compute the Bregman projection: $\bm\tau_{t+1} \leftarrow \Pi_{\mathcal{K}}(\bm\tau_{t+1})$.
        \end{enumerate}
        
 \item \textbf{Output}: $\bw \leftarrow$\textsc{Round}$\left(\bZ',\{\bw_j\}_{t=1}^T
     ,\theta\right)$ (or report \textsc{Fail} if \textsc{Round} fails).
       
\end{enumerate}
\end{mdframed}
    \caption{Approximate inner maximization via Bregman projection---\textsc{ApproxBregman}}
    \label{alg:bregman}
\end{figure}

Since Algorithm~\ref{alg:bregman} is  the MWU method with Bregman projections onto the set $\mathcal{K}$, we will apply the following regret guarantee.\footnote{To be more precise, the iterations $t$ in which $\|\bA_t \bw_t \|_2^2 \geq \frac{\theta^2}{10}$ behave according to the MWU method. Whenever $\|\bA_t \bw_t \|_2^2 < \frac{\theta^2}{10}$, the algorithm does not update the weights, which has no effect on the other iterations.}

\begin{theorem}[Theorem 2.4 of~\cite{arora2012multiplicative}] \label{thm:mwu-guarantee}
Suppose that for $\sigma^2_t (i) \in [0,1]$ for all $i \in [k']$ and $t \in [T]$. Then after $T$ iterations of Algorithm~\ref{alg:bregman}, for any $p \in \mathcal{K}$, it holds that:
\[
\sum_{t=1}^T \ip{\bm\tau_t}{\bm\sigma^2_t} \leq \frac{3}{2} \sum_{t=1}^T \ip{\bm p}{\bm \sigma^2_t} + 2 \KL(\bm p||\bm \tau_1).
\]
\end{theorem}

Finally, we comment that we cannot \naive ly apply the power method for the singular vector
computation. The power
method has failure probability of $\nicefrac{1}{10}$, whereas our algorithm should fail with
probability at most $\delta =O (\text{exp}(-k))$ that is exponentially low.
However, we note
that the algorithm  computes the top singular vectors of a sequence of
matrices $\bA_1, \bA_2, \ldots, \bA_T$. 
Observe that as long as $T = \Omega (\log (1/\delta)) = \Omega (k)$,
with probability at least $1-{\delta}/{8}$,
the power method will succeed for $0.9T$ of the matrices. 
We will show that this many successes suffice to guarantee correctness of our algorithm.

We first prove the following lemma, a requirement for the rounding algorithm to succeed.
\begin{lemma}[regret analysis] \label{lem:bregman-guarantee}
    After $T = O\left( \max \left( \frac{\log k'}{\theta^2}, \log (T_{\textsf{des}}/ \delta) \right) \right)$
iterations of Algorithm~\ref{alg:bregman}, for all but a ${1}/{4}$ fraction of $i \in [k']$:
\[
\sum_{t=1}^T \ip{\bZ'_i}{\bw_t}^2 \geq 100 \log k'.
\]
\end{lemma}
\begin{proof}
Let $S = \{i \in [k'] : |\ip{\bZ'_i}{\bw^*}| \geq \theta\}$ be the set of constraints satisfied by the unit vector $\bw^*$ whose existence is guaranteed in the hypothesis of~\autoref{lem:main}. By assumption, we have that $|S| \geq 0.8{k'}$. We simply calculate each of the terms in~\autoref{thm:mwu-guarantee}. 

First,  let $\mathcal{I} = \{t \in [T]: \text{$\bw_t$ is a $1/2$-approximate top singular vector of $\bA_t$}\}$. Then we have for any $t \in \mathcal{I}$: 
\begin{align*} 
\ip{\btau_t}{\bsig^2_t} &=  \sum_{i=1}^{k'} \bm\tau_t(i) \ip{\bZ'_i}{\bw_t}^2 && \text{(by definition)}\\ 
 &\geq  \frac{1}{2} \sum_{i=1}^{k'} \bm\tau_t(i) \ip{\bZ'_i}{\bw^*}^2 &&\text{(because $\bw_t$ is an approximate top eigenvector)} \\
 &\geq \frac{1}{2} \sum_{i \in S} \bm\tau_t(i) \ip{\bZ'_i}{\bw^*}^2 \\
 &\geq \frac{1}{2} \sum_{i \in S} \bm\tau_t(i) \theta^2 && \text{(by definition of $S$)} \\
 &\geq \frac{1}{2} \cdot \frac{1}{5} \theta^2 = \frac{\theta^2}{10} && \text{(because $|S| \geq 0.8k'$ and $\bm\tau_t \in \mathcal{K}$).}
\end{align*}
Summing this inequality over $t \in [T]$, we have that 
\[
\sum_{t=1}^T \ip{\btau_t}{\bsig ^2_t} \geq \sum_{t \in \mathcal{I}} \ip{\btau_t}{\bsig ^2_t} \geq  \frac{|\mathcal{I}|}{10} \theta^2.
\]
By  Chernoff-Hoeffding bound combined with the guarantee of power iteration (\autoref{fact:power}), as long as $T = \Omega
(\log (T_{\textsf{des}} / \delta ))$, then with probability at least $1 - \frac{\delta }{8
T_{\textsf{des}}}$, for at least $\frac{4}{5} T$ iterations, it will be the case that $\bw_t$ is an approximate top singular vector. In other words, $|\mathcal{I}| \geq \frac{4}{5} T$, so that we have:
\[
\sum_{t=1}^T \ip{\btau_t}{\bsig ^2_t} \geq \frac{2T}{25} \theta^2.
\]

Next, note that if we choose $\bp = \be_i$, then 
\[
\sum_{t=1}^T \ip{\bp}{\bsig^2_t} = \sum_{t=1}^T \ip{\bZ'_i}{\bw_t}^2.
\]
Because $\bm \tau_1$ is uniform, the relative entropy term in~\autoref{thm:mwu-guarantee} is at
most $\log k'$. Let's pretend for a moment that $\bm e_i \in \mathcal{K}$ (it is not). 
Then after plugging in the above calculations to~\autoref{thm:mwu-guarantee} and rearranging, we have that for every $i \in [k']$
\[
 \sum_{t=1}^T \ip{\bZ'_i}{\bw_t}^2 \geq \frac{2T}{25} \theta^2 - 2 \log k' \geq 100 \log k',
\]
by setting $T \geq \frac{10^5 \log k'}{\theta^2}$. This gives the bound claimed in the statement of the lemma, but it remains to fix the invalid assumption that $\be_i \in \mathcal{K}$. To do so, we will construct, for most $i \in [k']$, another distribution $\bp' \in \mathcal{K}$ such that
\[
\sum_{t=1}^T \ip{\be_i}{\bsig^2_t} \geq \sum_{t=1}^T \ip{\bp'}{\bsig^2_t}.
\] 
Combining this with $\sum_{t=1}^T \ip{\bp'}{\bsig^2_t} \geq 100 \log k'$ gives the desired lower
bound, for most $i$. Write $\bm \alpha = \sum_{t=1}^T \bsig^2_t$, and without loss of generality assume that 
\[
\bm \alpha_1 \geq \bm \alpha_2 \geq \ldots \geq \bm \alpha_{k'}.
\]
For $i = 1, \ldots, {4k'}/{5}$, take $\bp'$ to be uniform on those $j \in [k']$ such that $\bm
\alpha_i \geq \bm \alpha_j$ (there are at least ${k'}/{5}$ such $i$). By construction, we
have that $\ip{\bm \alpha}{\be_i} \geq \ip{\bm \alpha}{\bp'}$. Finally, observe that $\bp' \in
\mathcal{K}$ because $\bp'$ is uniform on a set of size at least ${k'}/{5}$.
\end{proof}

Observe that the \textsc{ApproxBregman} produces a sequence of vectors by the end. \cite{karnin2012unsupervised} provides a rounding algorithm that combines them into one with the desired margin bound. We describe the algorithm and prove the following lemma in Appendix~\ref{sec:obreg}.
\begin{lemma}\label{cor:fullround}
The algorithm \textsc{Round} (Algorithm~\ref{alg:round}) outputs $\bw$ that satisfies $|\ip{\bZ'_i}{\bw}| \geq 0.1\theta$ for  $0.45k$ of the points, with probability at least $1-\delta/4T_{\textsf{des}}$.
\end{lemma}
Finally, we are now ready to prove the key lemma using \textsc{ApproxBregman}.
\begin{proof}[Proof of~\autoref{lem:main}]
The correctness follows from~\autoref{cor:fullround}. We focus on run-time. By~\autoref{ass:scale}, we have that $1/\theta^2 = O(d)$. By projecting onto the subspace spanned by the
bucket means, we can assume $d \leq k$. Hence, \autoref{lem:bregman-guarantee} implies that the
iteration count is $\widetilde O(k')$. The runtime of each iteration is bounded by the cost of computing an 
approximate top singular vector of a $k'$ by $d$ matrix via the power method, which is $\widetilde O(k'd)$
by~\autoref{fact:power}. Finally, each repetition of the rounding algorithm \textsc{Round} takes time
$\widetilde O(k'd)$, and the number of trials is at most $O(\log (1/\delta'))$ by definition. Thus, the
runtime of the rounding algorithm is $\widetilde O(k^2d)$ .
\end{proof}

\subsection{Putting it Together}\label{sec:tgt}

Our main algorithm begins with the initial guess as the coordinate-wise median-of-means of $\{\bZ_i\}_{i=k+1}^{2k}$. Then it proceeds via the \textsc{Descent} procedure, where the gradient and  distance estimates are given by \textsc{ApproxBregman}. To ensure independence, we only use the  $\{\bZ_i\}_{i=1}^{k+1}$ for the descent part. We provide the full description in Appendix~\ref{sec:code}.

 We now give a proof sketch our main theorem. The formal proof is found in Appendix \ref{sec:omain}.
\begin{proof}[Proof sketch of~\autoref{thm:most}]
Our argument is conditioned on (i) that the Lugosi-Mendelson condition holds, (ii)  that the initial guess $\x_0$ satisfies an error bound $\sqrt{kd\|\bSig\|/n}$, and (iii) that  the \textsc{Prune} procedure succeeds. Each fails with probability at most $\delta /8$.

The guarantee of \textsc{ApproxBregman}, along with~\autoref{cor:estimate}, implies that \textsc{GradEst} and \textsc{DistEst} succeed with probability at least $1-\delta/4T_{\textsf{des}}$ each iteration. Taking union bound over all above events, the failure probability of the final algorithm is at most $\delta$. Applying the guarantee of the \textsc{Descent} procedure and error bound of the initial guess finishes the proof.
\end{proof}

%% file: discuss.tex
\section{Conclusion and Discussion} \label{sec:final}

In this paper, we provided a faster algorithm for estimating the mean of a heavy-tailed random vector that achieves subgaussian performance. Unlike previous algorithms, our faster running time is achieved by the use of a simple spectral method that iteratively updates the current estimate of the mean until it is sufficiently close to the true mean.

Our work suggests two natural directions for future research. First, is it possible to achieve subgaussian performance for heavy-tailed \emph{covariance} estimation in polynomial time? Currently, the best polynomial-time covariance estimators do not achieve the optimal statistical rate (see~\cite{lugosi2019survey, sam19}), while a natural generalization of the (computationally intractable) Lugosi-Mendelson estimator is known to achieve subgaussian performance. One approach would be to build on our framework; the key technical challenge is to design an efficient subroutine for producing bi-criteria approximate solutions to the natural generalization of the inner maximization problem to the covariance setting. 

Another direction is to achieve a truly linear-time algorithm for the mean estimation problem. Our iterative procedure for solving the inner maximization problem take $\widetilde O(k)$ iterations; is it possible to reduce this to a constant?

%% file: code.tex
\section{Main algorithm description}\label{sec:code}
\begin{figure}[htbp]
    \centering\begin{mdframed}[style=algo] 
\begin{enumerate}
    \item \textbf{Input:}  Dataset $\bZ'$ and  current estimate
        $\x_t$
    \item  $\bZ_i' \leftarrow \left(\bZ_i'-\x_t\right) / B$ by scaling each point by $B =\max_i
        \|\bZ_i' - \x_t \| $.
    \item  $\theta\leftarrow$  the largest margin $\theta$ such that \textsc{ApproxBregman}$\left(\bZ', \theta, T\right)$ does not \textsc{Fail}, where $T = O(\log k / \theta^2)$.
    \item \textbf{Output:} $\widehat{d} = \tfrac{B}{10}\theta $.
\end{enumerate}
\end{mdframed}
\caption{Distance estimation---\textsc{DistEst}}
\label{alg:dist}
\end{figure}
\begin{figure}[htbp]
    \centering\begin{mdframed}[style=algo] 
\begin{enumerate}
    \item \textbf{Input:}  Dataset $\bZ'$ and   current estimate $\x_t$ 
    \item  $\bZ_i' \leftarrow \left(\bZ_i'-\x_t\right) / B$ by scaling each point by $B =\max_i
        \|\bZ_i' - \x_t \| $.
    \item  $\theta\leftarrow$  the largest margin $\theta$ such that \textsc{ApproxBregman}$\left(\bZ', \theta, T\right)$ does not \textsc{Fail}, where $T = O\left(\log k / \theta^2\right)$.
    \item  $\widehat \bg \leftarrow \textsc{ApproxBregman}\left(\bZ', \theta, T\right)$
    \item  If $\langle \widehat \bg, \bZ_i'\rangle \geq 0.1\theta$ for at least $0.5k$ of the $\bZ_i'$, \textbf{output}  $\widehat \bg$; otherwise, \textbf{output}   $-\widehat \bg$.
\end{enumerate}
\end{mdframed}
\caption{Gradient estimation---\textsc{GradEst}}
\label{alg:grad}
\end{figure}
\begin{figure}[htbp]
    \centering\begin{mdframed}[style=algo] 
\begin{enumerate}
    \item \textbf{Input:}  Dataset $\bX_1,\bX_2,\cdots, \bX_n \in \R^d$ 
    \item Let $k = 3600\log (1/\delta)$. Divide the data into $2k$ groups. 
        \item Compute the bucket mean of each group: $\bZ_1,\bZ_2,\cdots,\bZ_{2k}\in \R^d$.
        \item Compute the coordinate-wise median-of-means of the second half of bucket means:
            $$\x_0 \leftarrow \textsc{MedianOfMeans}(\{ \bZ_{k+1},\cdots,\bZ_{2k} \}).$$
        \item Prune the first half of bucket means, where $\bZ$ is the data matrix of $\{\bZ_i\}_{i=1}^k7$: 
        \[
            \bZ' \leftarrow \textsc{Prune}(\bZ, \x_0).
        \]
    \item $T_{\textsf{des}} \leftarrow \Theta(\log d), \eta  \leftarrow 1/8000$
        \item Run the main descent procedure: $\widehat \bmu \leftarrow \textsc{Descent}(\bZ',
            \x_0, T_{\textsf{des}} , \eta)$, using \textsc{DistEst} and \textsc{GradEst} as above.
    \item \textbf{Output:} $\widehat\bmu $
\end{enumerate}
\end{mdframed}
\caption{Final algorithm}
\label{alg:whole}
\end{figure}

%% file: apx_fact.tex
\section{Technical facts} \label{apx:fact}

We formally state the statistical guarantee of empirical average and coordinate-wise
median-of-means. The former is an application of the Chebyshev's inequality. The latter is folklore but can follow easily from the Lugosi-Mendelson condition by considering the
projections onto standard basis vectors. 
\begin{lemma}[empirical mean] \label{lem:avg}
    Let $\delta \in (0,1)$. Given $n$ i.i.d.\ copies $\bX_1, \ldots, \bX_n$ of a random vector $\bX \in \R^d$ with 
mean $\bmu $ and covariance $\bSig$, let
$\overline \bmu = \tfrac{1}{n} \sum_{i=1}^n \bX_i$. Then  with probability at least $1-\delta$,
 \[
\|\overline \bmu - \bmu \| \leq \sqrt{ \frac{\Tr (\bSig)}{\delta n}}.
 \] 
\end{lemma}
\begin{lemma}[coordinate-wise median-of-means]\label{lem:gm}
Assume the Lugosi-Mendelson condition (\autoref{asp:key}). Let $\{\bZ_i\}_{i=1}^k $ be the bucket means from $n$ points (with at most $k/200$ contaminated) and $\widehat\bmu $ be their coordinate-wise
median-of-means. Then
with probability at least $1-\delta/8$,
\begin{equation*}
    \|\widehat{\bm{\mu}}  - \bm{\mu} \| \leq 600\sqrt{d} r_\delta \lesssim    \sqrt{\frac{d\|\bSig\| \log (1/\delta)}{n}}.
\end{equation*}
\end{lemma}
Our algorithm  requires
computing an approximation of 
the top (right) singular vector of a matrix $\bA \in \R^{m \times n}$.
The classic power method is efficient for this task.
\begin{fact}[power iteration; see Theorem 3.1 of~\cite{blumfoundations}] \label{fact:power}
Let $\lambda (\bA) = \max_{\bbx \in \mathbb{S}^{n-1}} \|\bA \bbx \|_2^2$.
With probability at least ${9}/{10}$, the power method (with random initialization) outputs a unit vector $\bw$ such that $\|\bA \bw \|_2^2 \geq \frac{\lambda (\bA)}{2}$ in $O(\log n)$ iterations. Moreover, each iteration can be performed in $O(mn)$ time.
\end{fact}
The following  is a standard  bound on binomial tail.
\begin{lemma}[\cite{okamoto1959some}]\label{lem:hoef}
    Let $H(n,p)$ be a binomial random variable. Then
    \[
        \Pr \left(H(n,p) \geq 2np\right)\leq\exp\left(-np /3\right).
    \]
\end{lemma}

%% file: omit.tex
\section{Omitted proofs from~\autoref{sec:gd}}\label{sec:gdo}
\begin{proof}[Proof of~\autoref{lem:descent-thm}]
First, suppose that in some iteration $t$ it holds that $\|\bmu - \x_t\| \leq 14000 r_\delta$. Then 
\[
\frac{1}{21} \|\bmu - \x_{t^*}\| \leq d_{t^*} \leq d_t \leq 2 \|\bmu - \x_t \| \leq 28000 r_\delta,
\]
so that we may conclude $\|\bmu - \x_{t^*}\| \leq 588000r_\delta$.
Second, suppose that in all iterations $t$ it holds that $\|\bmu - \x_t\| > 14000 r_\delta$. Then
by the update rule with $\eta = \nicefrac{1}{8000}$, 
\begin{align*}
    \|\x_{t+1} - \bmu \|^2 &= \|\x_{t} - \bmu \|^2 + 2 \eta d_t \ip{\x_{t} - \bmu}{\bg_t} + \eta^2 d_t^2 \|\bg_t\|^2 \\
                           &\leq \|\x_{t} - \bmu \|^2-  \frac{1}{800000} d_t \|\bmu -\x_t\| +
    \frac{1}{16000000}\|\x_{t} - \bmu \|^2 \\
                           &\leq \|\x_{t} - \bmu \|^2-  \frac{1}{1680000} \|\bmu -\x_t\|^2 +
    \frac{1}{16000000}\|\x_{t} - \bmu \|^2 \\
    &= \left(1-\frac{179}{336000000} \right) \|\x_{t} - \bmu \|^2
\end{align*}
Hence, the error bound drops at a geometric rate. The conclusion follows since $\|\bmu - \x_0 \| \leq
O(\sqrt{kd\|\bSig\|/n}) \leq  O(\sqrt{d} r_\delta)$. 
\end{proof}

\section{Omitted proof from~\autoref{sec:inner_max}}\label{sec:oinner_max}

 First recall that \cite{cherapanamjeri2019fast} showed  that the optimal solution to $\mathcal{M}( \bbx_t, \bZ)$  satisfies the
property that $\theta$ is a valid distance estimate (\cref{def:distest}) and $\bw $  a gradient estimate (\cref{def:gradest}).  
\begin{lemma}[Lemma 1 of~\cite{cherapanamjeri2019fast}]\label{lem:exact-dist}
For all $t= 1,2,\cdots, T$, let $d_t=\theta^*$ be  the optimal value of $\mathcal{M}(\x_t, \bZ)$.
Then $\left|d_t - \| \bm{\mu} - \x_t\|\right| \leq 600r_\delta$, so $d_t$ is a distance
estimate  with respect to $\x_t$.
\end{lemma}
\begin{lemma}[Lemma 2 of~\cite{cherapanamjeri2019fast}]
For all $t= 1,2,\cdots, T$, let $(\theta^*,\bb^*,\bw^*)$ be the optimal solution of
$\mathcal{M}(\x_t , \bZ)$. Then $\bg_t$ is a distance estimate with
respect to $\x_t$.  
\end{lemma}

We now start by proving a generic claim that any reasonably good \textit{bicriteria approximation} of $\mathcal{M}( \bbx_t, \bZ)$ suffices to provide gradient and distance estimates.

\begin{definition}[bicriteria solution]
Let $ \theta^* $ be the optimal value  of $\mathcal{M}(\bbx, \bZ)$.  We say that $(\theta, \bb,\bw)$ is a $(\alpha,\beta)$-\textit{bicriteria  solution} to $\mathcal{M}(\bbx, \bZ)$ if $\sum_{i} b_i \geq \alpha k$ and $b_i \ip{\bZ_i - \x}{\bw} \geq b_i   \theta$ for all $i$, where $\theta = \beta \theta^*$.
\end{definition}

\begin{lemma}[distance estimate]\label{lem:dist-est}
Let $(\theta, \bb,\bw)$ be a $(1/10, 1/20)$-bicriteria solution to $\mathcal{M}( \bbx_t , \bZ)$.
Then $d_t = \theta  $ is a distance estimate with respect to $\x_t$. 
\end{lemma}
\begin{proof}[Proof of~\autoref{lem:dist-est}]
 By~\autoref{lem:exact-dist}, the optimal value $\theta^*$ lies in the range \[\left[  \| \bmu -
 \x_t \| - 600r_\delta,   \| \bmu - \x_t \| + 600r_\delta\right].\] Moreover, since $\theta^*/20 \leq \theta \leq \theta^* $, we have that 
 \begin{equation}\label{eqn:dis-ex}
    \frac{\|\bmu - \x\|}{20} - 30r_\delta \leq \theta \leq   \frac{\|\bmu - \x\|}{20} + 30r_\delta.
\end{equation}
\begin{itemize}
    \item  When $ \|\bmu -\x\| \geq 14000r_\delta $, we get from the inequality~\eqref{eqn:dis-ex} that 
\begin{equation*}
    \frac{\|\bmu - \x\|}{21}  \leq \theta \leq   \frac{\|\bmu - \x\|}{19}.
\end{equation*}
\item When  $\|\bmu -\x\| \leq 14000r_\delta$, $\theta \leq 730r_\delta < 28000r_\delta$, again
    by~\eqref{eqn:dis-ex}.
\end{itemize}
\end{proof}
 
\begin{lemma}[gradient  estimate] \label{lem:grad-est}
Let $(\theta, \bb,\bw)$ be a $(1/10, 1/20)$-bicriteria solution to $\mathcal{M}( \bbx_t , \bZ)$.
Then $\bg_t = \bw$ is a gradient estimate with respect to $\x_t$. 
\end{lemma} 

\begin{proof}[Proof of~\autoref{lem:grad-est}]
Let $\bg^* = (\bm{\mu} - \x_t)/ \| \bm{\mu} - \x_t \|$ be the true gradient. We need to show that $\ip{\bg^*}{\bg_t} \geq 1/20$. On the one hand, by~\autoref{lem:exact-dist}, we have 
\begin{equation}\label{eq:dt1}
    d_t = \theta \geq \frac{1}{20} ( \|\bu -\x_t\| - 600r_\delta).
\end{equation}
On the other hand, for  at least $k/10$ points, we have  $\ip{\bZ_i -\x_t}{\bg_t} \geq d_t$
and for at least  $0.95k$ points,
we have $\ip{\bZ_i -\bmu}{\bg_t} \leq 600r_\delta$  by~\autoref{asp:key}.
Hence, there must be a point $\bZ_j$ that satisfies both inequalities, so it follows that
\begin{equation}\label{eq:dt2}
     d_t \leq \ip{\bZ_j - \x_t}{\bg_t}
     = \ip{\bZ_j - \bm{\mu}}{\bg_t} + \ip{\bm{\mu} - \x_t}{\bg_t}
     \leq 600r_\delta + \| \bm{\mu} - \x_t \| \ip{\bg^*}{\bg_t}.
\end{equation}
Using~\eqref{eq:dt1} and~\eqref{eq:dt2} and rearranging,
\[
 \ip{\bg^*}{\bg_t} \geq \frac{1}{20} - \frac{630r_\delta}{\|\bmu -\x_t\|} \geq \frac{1}{200},
\]
where we use   $\|\bmu - \x_t\| \geq 14000r_\delta$.
\end{proof}

Now we show that the optimal solution to the two-sided relaxation give distance and gradient estimate.

\begin{lemma}\label{lem:relax}
Let $(\theta',\bb',\bw')$ be an optimal solution of $\mathcal{M}_2(\x, \bZ)$. We have that 
\begin{enumerate}[(i)]
    \item the value $\theta'$ lies in $\left[  \| \bmu - \x \| - 600r_\delta,   \| \bmu - \x \| + 600r_\delta\right]$; and
    \item one of the following two statements must hold, if $\|\bmu - \x\|\geq 14000r_\delta$:
    \begin{itemize}
        \item  there is a set $\mathcal{C}$ of at least $ 0.9k$ points such that $\ip{\bZ_i - \x}{\bw'} \geq   \theta'  $ for all $i\in \mathcal{C}$; or 
        \item there is a set $\mathcal{C}$ of at least $ 0.9k$ points such that $\ip{\bZ_i - \x}{-\bw'} \geq   \theta'  $ for all $i\in \mathcal{C}$.
\end{itemize}
\end{enumerate}
\end{lemma}
\begin{proof}[Proof of~\autoref{lem:relax}]
Let $ \theta $ be the optimal value of $\mathcal{M}(\x, \bZ)$. To prove (i), first recall that~\autoref{lem:exact-dist} states that  $\theta \geq \|\bmu - \x \|- 600r_\delta$. Therefore, we get that $\theta' \geq \|\bmu - \x \|- 600r_\delta$, as $\theta' \geq \theta$.  For the upper bound, assume for the sake of a contradiction that $\theta'  >  \|\bmu - \x \| + 600r_\delta$. Then one side of the hyperplane defined by $\bw '$ must contain at least $19/40$ fraction of points, so let's suppose without loss of generality that 
\begin{equation}\label{eq:con1}
    \ip{\bZ_i - \x}{\bw'} \geq  \theta' >  \|\bmu - \x \| + 600r_\delta
\end{equation}
for at least $19k/40$ $\bZ_i$'s. Also, note that 
\begin{equation}\label{eq:con2}
     \ip{\bZ_i - \x}{\bw'}  = \ip{\bZ_i -\bmu}{\bw'} + \ip{\bmu - \x}{\bw'} \leq \|\bmu -\x\| + \ip{\bZ_i -\bmu}{\bw'}.
\end{equation}
Combining~\eqref{eq:con1} and~\eqref{eq:con2}, it follows that for at least $19k/40$ $\bZ_i$'s we have  
\begin{equation}\label{eq:conclu}
    \ip{\bZ_i - \bmu}{\bw'} > 600r_\delta.
\end{equation}
On the other hand, consider projections of all bucket means $\bZ_i $ onto $\bw'$. \autoref{asp:key} implies that 
\begin{equation*}
   \left | \left\{ i :    \ip{\bw'}{\bZ_i  } - \ip{\bw'}{\bmu  }    \geq 600r_\delta)\right\}  \right|  \leq 0.05k.
\end{equation*}
This means that at most $k/20$ points satisfy $\ip{\bZ_i-\bmu}{\bw'} \geq 600r_\delta$, contradicting~\eqref{eq:conclu}.

To prove (ii), let $S^+ = \{ i: \ip{\bZ_i - \x}{\bw'} \geq   \theta' \}$ and $S^- = \{ i: \ip{\bZ_i - \x}{-\bw'} \geq   \theta' \}$. Notice that since $\|\bmu - \x\| \geq 14000r_\delta$, 
$S^+$ and $S^-$ are disjoint.   Now let
\begin{equation*}
    B = \left\{ i :   | \ip{\bw'}{\bZ_i -\bmu}|     \leq 600 r_\delta \right\} = \left\{ i :   | \ip{\bw'}{\bZ_i -\x}  -  \ip{\bw'}{\bmu -\x}|     \leq 600 r_\delta \right\}.
\end{equation*}
By~\autoref{asp:key}, $|B| \geq 19k/20$. Consider the two cases.
\begin{itemize}
    \item If $ \ip{\bw'}{\bmu -\x} \geq 0 $, observe that $B$ must intersect $S^+$ but not $S^-$.  This implies that $|S^-| \leq k/20$, so $|S^+| \geq  9k/10$, since $|S^+|+|S^-|=19k/20$ and they are disjoint.
    \item If $ \ip{\bw'}{\bmu -\x} < 0 $, by the same argument, we have $|S^-| \geq 9k/10$. 
\end{itemize}
\end{proof}

Next, we show that approximating $\mathcal{M}_2$ in a bicriteria manner 
achieves a similar guarantee.%
\begin{lemma}\label{lem:biM}
Let $\theta^*$ be the optimal value of $\mathcal{M}(\x, \bZ)$ and $\bw'$ 
be a unit vector such that for at least $k/8$ of the $\bZ_i$, we have $|  \ip{\bw'}{\bZ_i -\x}  | \geq \theta'$, 
where $\theta '=  0.1\theta^*$. One of the following two statements must hold if $\|\bmu - \x\|\geq 14000r_\delta$.
    \begin{itemize}
        \item  there is a set $\mathcal{C}$ of at least $0.95k$ points such that $\ip{\bZ_i - \x}{\bw'} \geq   \theta' -600r_\delta $ for all $i\in \mathcal{C}$; 
        \item there is a set $\mathcal{C}$ of at least $ 0.95k$ points such that $\ip{\bZ_i - \x}{-\bw'} \geq   \theta' -600r_\delta $ for all $i\in \mathcal{C}$.
    \end{itemize}
\end{lemma}

\begin{proof}[Proof of~\autoref{lem:biM}]
Let $\mathcal{C} = \{ i :  |  \ip{\bw'}{\bZ_i - \bmu }| \leq 600r_\delta \}$ be the set of ``good'' points with respect to direction $\bw'$. 
By~\autoref{asp:key}, $|\mathcal{C}|\geq 19k/20$. Further, let $S = \{| \ip{\bw'}{\bZ_i -\x}  | \geq \theta'\} $, which we assume has size at least $k/8$. Thus, by pigeonhole principle, there must be a point, say $\bZ_j$, that is in both sets.  There are two cases. 
\begin{itemize}
    \item Suppose $\ip{\bw'}{ \bmu  -\x} \geq 0$. Since $j\in S$ and $\theta^* \geq 13400r_\delta$ by~\autoref{lem:relax},   we have $| \ip{\bw'}{\bZ_i -\x}  | \geq  1340r_\delta$. On the other hand,  since $j \in \mathcal{C}$,
    \begin{equation}
        \left|  \ip{\bw'}{\bZ_j - \bmu }\right|  = \left|  \ip{\bw'}{\bZ_j - \x }-    \ip{\bw'}{\bmu- \x }\right|\leq 600r_\delta.
    \end{equation} 
    Hence, we observe that $\ip{\bw'}{\bZ_j -\x}    \geq \theta' \geq 1340r_\delta$. By definition of $\mathcal{C}$, all its points cluster around $\bZ_j$ by an additive factor of $600r_\delta$. 
    \item Suppose $\ip{\bw'}{ \bmu  -\x} \leq 0$. We get the second case in the claim by the same argument.
\end{itemize}
\end{proof}

Finally, we are ready to prove~\autoref{cor:estimate}.
 
\begin{proof}[Proof of~\autoref{cor:estimate}]
    Let's first check the distance estimate (\autoref{def:distest}) guarantee. 
\begin{itemize}
    \item If $\|\bmu -\x\|\geq 14000r_\delta$,  we have
\begin{equation*}
    \theta'  \geq \frac{1}{10}  \| \bmu - \x \| - 60r_\delta \geq \frac{2}{35}\|\bmu - \x\|,
\end{equation*}
since $\theta'  = 0.1\theta^*$ and $\theta^* \geq \| \bmu - \x \| - 600r_\delta$.
The upper bound of~\eqref{cond:dist-est} obviously holds.
\item If   $\|\bmu -\x\|\leq 14000r_\delta$, we have $\theta' \leq 1460r_\delta$
    by~\autoref{lem:relax}.
\end{itemize} 
For gradient estimate, we appeal to~\autoref{lem:biM} and get that if $\|\bmu -\x\|\geq 14000r_\delta$,  then either  $(\theta', \bb',\bw')$ or $(\theta', \bb',-\bw')$ is a $(19/20, 1/20)$-bicriteria approximation of $\mathcal{M}(\x,\bZ)$, where $\bb'$ is the indicator vector of $\mathcal{C}$. Thus, we can apply~\autoref{lem:grad-est}, and this completes the proof.
\end{proof}

\section{Omitted proof from~\autoref{sec:prune}} \label{apx:prune}

We remark that under Lugosi-Mendelson condition, the assumption $\|\bmu - \x_0\| \lesssim \sqrt{ kd \|\bSig\|/n}  $ can be easily achieved by
initializing $\x_0$ to be the coordinate-wise median-of-means (\autoref{lem:gm}) (with a failure
probability at most $\delta/8$).
\begin{lemma}[pruning] \label{lem:prune}
Let $\beta =600\sqrt{ kd \|\bSig\|/n}$, and suppose $\|\bmu - \x_0\|
\leq \beta$. Given the bucket means $\bZ \in \R^{k\times d}$ such that at most $k/200$ points are contaminated, the
algorithm \textsc{Prune} removes $k/10$ of the points and guarantees that with probability at least
$1-\delta/8$, among the remaining data, $$\max_i \|\bZ_i - \bm \mu\| \leq O(\beta ).$$
Further, \textsc{Prune}$(\bZ, \x_0)$ can be implemented in $\widetilde O (kd)$ time.
\end{lemma}
\begin{proof}[Proof of~\autoref{lem:prune}]
For correctness, consider $\|\bZ_i -\bmu\|$, and by triangle inequality,
\begin{equation*}
    \|\bZ_i -\x_0\| - \|\bmu - \x_0\| \leq \| \bZ_i - \bmu \| \leq \|\bZ_i -\x_0\| + \|\bmu -
    \x_0\|.
\end{equation*}
Since $\| \bmu -\x_0 \| \leq \beta$ by our assumption, 
\begin{equation}\label{eqn:badgood}
    \|\bZ_i -\x_0\| - \beta \leq \| \bZ_i - \bmu \| \leq \|\bZ_i -\x_0 \| + \beta.
\end{equation}
Let $\mathcal S_\textsf{good} = \{ i : \|\bZ_i - \bmu \| \leq \beta\}$ and
$\mathcal{S}_\textsf{bad} = \{ i : \|\bZ_i - \bmu \| \geq 20\beta\} $. It suffices to show that
with probability at least $1-\delta/8$  all the points in $\mathcal S_{\textsf{bad}}$ are removed.  
We first lower bound the number of good points.
Each uncontaminated $\bZ_i$ is an average of $\lfloor n/k\rfloor $ i.i.d.\ random vectors.
Applying~\autoref{lem:avg} on estimation error of empirical mean, we obtain that for each uncontaminated $i$, 
with probability at
least $1-{1}/{1000}$,
\begin{equation*}
    \| \bZ_i - \bmu \| \leq \sqrt{ 1000\cdot {\Tr(\bSig) k}/{n} } \leq \beta.
\end{equation*}
Therefore, each
uncontaminated $\bZ_i$ is in $\mathcal S_{\textsf{good}}$ with probability at least $1-{1}/{1000}$.
Let $H$ be the number of uncontaminated points not in $\mathcal S_{\textsf{good}}$ and $p={1}/{1000}$.
Since there are at most $k/200$ contaminated points and each uncontaminated point is independent, by a binomial tail bound
(\autoref{lem:hoef})
\begin{align}
    \Pr \left(H \leq 2p \cdot (199/200)k\right)  &\geq 1- \exp \left(-p \cdot (199/200)k /3 \right)\nonumber \\
    &\geq  1-\exp\left(-\log \left(8/\delta\right)\right)\nonumber\\
    &= 1- \delta / 8\nonumber,
\end{align}
where we used $k= \lceil 3600\log (8/\delta)\rceil$. Hence, with probability at least $1-\delta/8$, $\mathcal
S_\textsf{good}$ contains
at least $({399}/{400})k$ (uncontaminated) points.  We condition on this event for the rest of the proof.

Now observe that \begin{equation}\label{eq:gbad}
    \|\bZ_i -\x_0 \|  < \|\bZ_j -\x_0 \| \,\,\,\,\, \text{for each $j\in \mathcal S_{\textsf{bad}}$ and
$i \in  \mathcal S_{\textsf{good}}$}
\end{equation}
by~\eqref{eqn:badgood}. Suppose for a contradiction
that $j\in \mathcal S_{\textsf{bad}} $ is not removed by line 4. Then it means that  $d_j \leq
d_i$ for $k/10$ of the $\bZ_i'$'s. By pigeonhole principle, this implies $d_j \leq d_i$ for some $i
\in  \mathcal S_{\textsf{good}}$, since $|S_\textsf{good}| \geq (399/400)k$. This contradicts
condition~\eqref{eq:gbad}.

Computing the distances takes $O(kd)$ time and sorting takes $O(k\log k)$ time. Thus, the algorithm \textsc{Prune} runs in time $O(kd + k \log k)$ and succeeds with probability at least $1-\delta/8$.  
\end{proof}

Pruning allows us to bound the norms of  the points $\bZ_i - \x_t$ for each iteration $t$.  
 
\begin{corollary}[scaling and margin]\label{cor:scale}
    Suppose $\|\bmu - \x\|\leq O \left(\sqrt{ {\|\bSig\| kd}/{n} }\right)$ and $\|\bmu - \x\|\geq
    \Omega\left(r_\delta\right)$. 
    Let $\mathcal S$ be the pruned dataset of size  $k'\geq 9k/10$
such that $\|\bZ_i-\bmu \| \leq  O \left(\sqrt{ {\|\bSig\| kd}/{n} }\right)$ for each
$i\in\mathcal S$. 
    There exists a   scaling factor $B$,  $\theta>0$ and unit vector $\bw$ such that for at least
    $4k/5$  points in $\mathcal S$,  
    \[
    \left|\ip{\tfrac{1}{B} (\bZ_i - \x)}{\bw}\right| \geq   \theta.
    \]
Further, we have that $1/\theta^2 = O(d)$.  
\end{corollary}

\begin{proof}[Proof of~\autoref{cor:scale}] 
Let $B = \max_{i\in \mathcal S} \|\bZ_i - \x\|$. 
Then $B$  is bounded by
\begin{align}
       \| \bZ_i - \x \| \leq \|\bZ_i -\bmu\| + \|\bmu - \x\| 
       \leq O \left(\sqrt{ {\|\bSig\| kd}/{n} }\right).
\end{align}
By~\autoref{lem:relax},  there exists a unit vector $\bw$ such that for at least $0.8k $   points
in $\mathcal S$,  $\ip{\bZ_i - \x}{\bw} \geq   \theta'$ and $\theta' = \Omega (r_\delta)$.  Hence, we get that 
 \begin{equation*}
     \theta = \Omega\left(\frac{r_\delta}{B}\right) = \Omega\left( \frac{\sqrt{k\|\bSig\|/n} +
     \sqrt{\Tr\bSig /n }} {\sqrt{\|\bSig\| \cdot kd /n }}\right) =\Omega\left(1/ \sqrt{d}\right).
 \end{equation*}
\end{proof}

\begin{proof}[Proof of~\autoref{cor:2ass}]
The lemma follows directly from~\autoref{lem:prune} and~\autoref{cor:scale}.
\end{proof}

\section{Omitted proofs from~\autoref{sec:breg}}\label{sec:obreg}
\cite{karnin2012unsupervised} provides a rounding scheme that combines the sequence of vectors produced by \textsc{ApproxBregman} into one vector that satisfies the desired margin bound. The original routine succeeds with constant probability. We simply perform independent trials to boost the rate. 
\begin{figure}[ht]
      \centering\begin{mdframed}[style=algo] 
\begin{enumerate}
    \item \textbf{Input:} Buckets means $\bZ'$, unit vectors $\bw_1,\ldots,\bw_T \in \R^{  d}$,  margin $\theta$,
     \item Round to a single vector: $\bw = \frac{\bw'}{\| \bw' \|}$, where $\bw' = \sum_{t=1}^T g_t \bw_t$ and $g_t \sim \mathcal{N}(0,1)$, for $t =1,\ldots,T$.
    
    \item Repeat until $|\ip{\bZ'_i  }{\bw}| \geq \frac{1}{10} \theta$ for at least $0.6k'$ of $\bZ'_i $:
        \begin{enumerate}
            \item Sample $g_t \sim \mathcal{N}(0,1)$, for $t=1,\ldots,T$.
            \item Recompute $\bw = \bw' / \| \bw' \|$, where $\bw' = \sum_{t=1}^T g_t \bw_t$.
            \item Report \textsc{Fail} if  more than $\Omega \left(\log \left(T_{\textsf{des}}/\delta\right)\right)$ trials have been performed.
        \end{enumerate}
           \item \textbf{Output:} $\bw$ 
    \end{enumerate}
    \end{mdframed}
    \caption{Rounding algorithm---\textsc{Round}}
    \label{alg:round}
\end{figure}
We now analyze the algorithm. We cite the following lemma for the guarantee of the rounding algorithm (Algorithm~\ref{alg:round}).
\begin{lemma}[Lemma 6 of~\cite{karnin2012unsupervised}]\label{lem:round}
Suppose that for at least $\frac{3}{4}$ fraction of $i \in [k']$, it holds that 
\begin{equation}\label{eq:forround}
\sum_{t=1}^T \ip{\bZ'_i}{\bw_t}^2 \geq   \log k'.
\end{equation}
Let $\bw_1, \ldots, \bw_T$ be the unit vectors satisfying the
above condition. Then with constant probability, the vector $\bw$ in each repetition  
of the step 3 of the \textsc{Round} algorithm (Algorithm~\ref{alg:round}) satisfies $| \ip{\bZ_i}{\bw} | \geq  \theta/10 $ for at least a
$0.45$ fraction of $i \in [k']$.
\end{lemma}
Now we prove the guarantee of \textsc{Round}.
\begin{proof}[Proof of~\autoref{cor:fullround}]
By~\autoref{lem:round}, it suffices to prove inequality~\eqref{eq:forround} holds for at least a $3/4$ fraction
of the points. By the regret analysis (\autoref{lem:bregman-guarantee}), the vectors $\bw_1, \ldots, \bw_T$ produced during  the iterations of Algorithm~\ref{alg:bregman} satisfy the hypothesis of \autoref{lem:round}. Hence, the guarantee
of~\autoref{lem:round} holds with constant probability. Moreover, we can test that this guarantee holds in time $O(T k' d)$. To boost the success probability to $1-\delta'$ (with $\delta' = \delta / 4T_{\textsf{des}}$), \textsc{Round} algorithm performs
$ O(\log (1/\delta'))$ independent trials. Hence, it reports \textsc{Fail} with probably at most $\delta'$.  Otherwise, by its definition, the output $\bw$ satisfies 
desired bound $|\ip{\bZ'_i}{\bw}| \geq 0.1\theta$ for  $0.45k$ of the points.
\end{proof}

\section{Full proof of main theorem}\label{sec:omain}
\begin{proof}[Proof of~\autoref{thm:most}]
Our argument assumes  the following global events.
\begin{enumerate}[(i)]
        \item The Lugosi-Mendelson condition (\autoref{asp:key}) holds.
        \item The initial estimate $\bbx_0$ satisfies $\| \bmu -\bbx_0 \|\leq  600  
            \sqrt{\|\bSig\| kd/n}$.
        \item The pruning step succeeds: $\|\bZ_i' - \bmu \| \leq O\left(  \sqrt{\|\bSig\| kd/n}\right) $ 
\end{enumerate}
We consider our main algorithm (Algorithm~\ref{alg:whole}) and first prove the correctness of \textsc{DistEst} and \textsc{GradEst}.
Let $\bZ'$ be
    defined as in line 2 of  \textsc{DistEst} and \textsc{GradEst}. 
    \autoref{lem:relax} states that there exists a margin $\theta^*$  in $\left[  \| \bmu - \x \| - 600r_\delta,   \| \bmu - \x \| +
    600r_\delta\right]$. 
    When $\|\bm \mu -
    \bbx_t \| \geq 14000r_\delta$, we have that
    for at least $0.8k$ points $\bZ_i'$ it holds
    $B\cdot |\ip{\bZ_i'}{\bw^*}| \geq \theta^*$ for some unit vector $\bw^*$, since the data are scaled by $B$. 
    Furthermore, when the pruning step succeeds,~\autoref{asp:key} holds. 
    This allows us to apply the key lemma (\autoref{lem:main}). 
    \begin{enumerate}[(i)]
        \item 
    For \textsc{GradEst}, we use binary search in line 3 to find $\theta = \theta^*/ B$. 
    By the guarantee of~\autoref{lem:main}, $|\ip{\bw}{\bZ_i'}|\geq \tfrac{\theta}{10}$ for at
    least $\nicefrac{k}{8}$ of the $\bZ_i$. It follows that  $|\ip{\bw}{\bZ_i - \x_t }|\geq
    \tfrac{B\theta}{10}$ for at least $\nicefrac{k}{8}$ of the $\bZ_i$.   Thus, \autoref{cor:estimate}
    implies that the output $\bg_t$  is a gradient estimate.
\item By the same argument, we apply \autoref{cor:estimate}
 and conclude that $\widehat{d}_t$ of \textsc{DistEst} is a distance estimate. 
    \end{enumerate}
    Finally, we apply \autoref{lem:descent-thm} for the guarantee of \textsc{Descent}. 

    Next we bound several failure probabilities of the algorithm. The first three correspond to the global
    conditions.
\begin{itemize}
    \item By~\autoref{lem:lmm}, the Lugosi-Mendelson condition \autoref{asp:key} fails with probability
        at most $\delta/8$.
    \item By \autoref{lem:gm}, the coordinate-wise median-of-means error bound fails with probability at most
        $\delta/8$ 
    \item By~\autoref{cor:2ass}, the  guarantee of our
        pruning  and scaling procedure (\autoref{ass:scale}) fails with probability at most $\delta/8$.
    \item  Conditioned on above, the \textsc{ApproxBregman} satisfies the guarantee of the key lemma (\autoref{lem:main}). The failure probability is at most $\delta/4T_{\textsf{des}}$ each iteration. We take union bound over all these iterations. 
\end{itemize}
Overall, the failure probability of the entire algorithm (Algorithm~\ref{alg:whole}) is bounded
by $\delta$ via union bound.

The runtime follows from~\autoref{lem:main} which claims each iteration takes time $\widetilde O(k^2 d)$ and the fact that $T_{\textsf{des}} = \widetilde O(1)$.
\end{proof}

%% file: karnin.tex
\section{Interpretation of FHP algorithm \texorpdfstring{\cite{karnin2012unsupervised}}{} as regret minimization}
\label{sec:fhp-mwu}

Here, we review the bicriteria approximation algorithm of Karnin \etal~\cite{karnin2012unsupervised} and show how it can be interpreted in the multiplicative weights update (MWU) framework for regret minimization. Given $\bZ_1, \ldots, \bZ_k \in \R^d$ such that $\| \bZ_i \| \leq 1$, we study the following \emph{furthest hyperplane problem}:
\begin{align*}
\text{Find} \quad &\bw \in S^{d-1} \\
\text{subject to} \quad &|\ip{\bZ_i}{\bw}| \geq r \text{ for } i = 1,\ldots, k, 
\end{align*}
where we are promised that there does indeed exist a feasible solution $\bw^*$. Since this problem is (provably) hard (even to approximate) we will settle for \emph{bicriteria approximate} solutions. By this, we simply mean that we require the algorithm to output some $\bw$ such that $|\ip{\bZ_i}{\bw}| \geq \frac{r}{10}$ for most of the $i \in [k]$. For our applications, the particular constants will not matter much, as long as they are actually constants. 

See Algorithm~\ref{alg:fhp} for a formal description. First we give some intuition and then we sketch the important steps in the analysis. 
\begin{figure}[htbp]
    \centering\begin{mdframed}[style=algo] 
\begin{enumerate}
    \item \textbf{Input:} $\bZ_1,\ldots,\bZ_k \in \R^{d}$ and iteration count $T \in \N$.
    
    \item Initialize weights:  $\bm\tau_1 = \frac{1}{k} (1, \ldots,1) \in \R^k$.
    
    \item For $t=1,\ldots,T$, repeat: 
        \begin{enumerate}
            \item Let $\bA_t$ be the $k \times d$ matrix whose $i$th row is $\sqrt{\bm\tau_t(i)} \bZ_i$ and $\bw_t$ be the top right unit singular vector of $\bA_t$. 
            
            \item Set $\bm\sigma_t(i) = | \ip{\bZ_i}{\bw_t} |$.
            
            \item Reweight: $\bm\tau_{t+1}(i) =\bm \tau_t(i) \eta^{-\bm\sigma^2_t(i)}$ for $i \in [k]$ for an appropriately chosen constant $\eta$. In MWU language, $\bm\sigma^2_t$ is the loss vector at time $t$.
            \item Normalize: Let $Z = \sum_{i \in [k]} \bm\tau_{t+1}(i)$ and redefine $\bm\tau_{t+1} \leftarrow \frac{1}{Z} \bm\tau_{t+1}$.
            
        \end{enumerate}
        
    \item \textbf{Output:} $\bw_1, \ldots, \bw_T \in S^{d-1}$.
\end{enumerate}
\end{mdframed}
    \caption{Iterative MWU procedure}
    \label{alg:fhp}
\end{figure}

\subsection{Intuition}
Because we are promised that $\bw^*$ exists, averaging the constraints yields:
\[
\frac{1}{k} \sum_{i=1}^k \ip{\bZ_i}{\bw^*}^2 \geq r^2.
\]
Note that if we define $\bA_1$ as in Algorithm~\ref{alg:fhp}, then the definition of singular vector tells us that:
\[
\max_{\bw \in S^{d-1}} \| \bA_1 \bw \|^2 = \frac{1}{k} \sum_{i=1}^k \ip{\bZ_i}{\bw}^2 \geq \frac{1}{k} \sum_{i=1}^k \ip{\bZ_i}{\bw^*}^2 \geq r^2.
\]
Thus, $\bw_1$, the top singular vector as defined in Algorithm~\ref{alg:fhp}, satisfies the constraints \emph{on average}. It could be the case that $\ip{\bZ_4}{\bw_1}^2 \gg r^2$ but $\ip{\bZ_i}{\bw}^2 \ll r^2$ for all $i \neq 4$. To fix this issue, we would simply down-weight $\bZ_4$ in the next iteration, so that $\bw_2$ aligns more with $\bZ_i$ for $i \neq 4$. We repeat this several times, with each $\bw_t$ improving upon $\bw_{t-1}$. 

At the end, the algorithm produces a collection of vectors $\bw_1, \ldots, \bw_T$ which each satisfy a certain property. While it seems natural to just output $\bw_T$ as the final answer, it turns out that this will not work. Instead, we need to apply a randomized rounding procedure to extract a single vector $\bw$ from $\bw_1, \ldots, \bw_T$.

\subsection{Analysis}
\begin{lemma}
When Algorithm~\ref{alg:fhp} terminates after $T = O(\frac{\log k}{r^2})$ iterations, for every $i \in [k]$ it holds that: 
\[
\sum_{t=1}^T \ip{\bZ_i}{\bw_t}^2 \geq \frac{\log k}{\log \eta}.
\]
\end{lemma}
\begin{proof}
Algorithm~\ref{alg:fhp} is simply the MWU algorithm with the  experts corresponding to the $k$ constraints and the loss of expert $i$ at time $t$ being $\bm\sigma^2_t(i)$. Using the regret guarantee from Theorem 2.1 in~\cite{arora2012multiplicative} with respect to the fixed expert $\be_i$ and step size $\eta$:
\begin{equation}
    \sum_{t=1}^T \ip{\bm \tau_t}{\bm\sigma^2_t} - (1+ \eta) \sum_{t=1}^T \ip{\be_i}{\bm\sigma^2_t} \leq \frac{\log k}{\eta}.
\end{equation}
Note that
\[
\sum_{t=1}^T \ip{\be_i}{\bm\sigma^2_t} = \sum_{t=1}^T \ip{\bZ_i}{\bw_t}^2
\]
and
\begin{align*}
\sum_{t=1}^T \ip{\bm\tau_t}{\bm\sigma^2_t} &= \sum_{t=1}^T \sum_{i=1}^k \bm\tau_t(i) \bm\sigma^2_t(i) \\
&= \sum_{t=1}^T \sum_{i=1}^k \bm\tau_t(i) \ip{\bZ_i}{\bw_t}^2 && (\text{by definition of the algorithm})\\ 
&\geq \sum_{t=1}^T \sum_{i=1}^k \bm\tau_t(i) \ip{\bZ_i}{\bw^*}^2 && (\text{since $\bw^*$ is the top eigenvector})\\
&\geq \sum_{t=1}^T \sum_{i=1}^k \bm\tau_t(i) r^2 \\
&= T r^2.   
\end{align*}
Substituting into and simplifying the regret formula and taking $\eta = \nicefrac{1}{3}$ gives the claim.
\end{proof}
Given the previous lemma, we can just apply the rounding algorithm as a black-box to the output of Algorithm~\ref{alg:fhp}.
\begin{lemma}[\cite{karnin2012unsupervised}]
Let $\alpha \in (0,1)$ and $\bw_1, \ldots, \bw_T$ be unit vectors satisfying the conclusion of the previous lemma. Then with probability at least $1/147$, the output $\bw$ of the Rounding Algorithm~\ref{alg:round} satisfies $| \ip{\bZ_i}{\bw} | \geq \alpha r$ for at least a $1 - 3 \alpha$ fraction of $i \in [k]$.
\end{lemma}